\documentclass{amsart}
\usepackage{fix-cm}
\usepackage[final]{microtype}
\usepackage[dvipsnames,svgnames,x11names,hyperref]{xcolor}
\usepackage{dsfont,url,graphicx,verbatim,amssymb,enumerate,stmaryrd,booktabs,lmodern,mathtools,mathabx,nicefrac}
\SetSymbolFont{stmry}{bold}{U}{stmry}{m}{n}
\usepackage[pagebackref,colorlinks,citecolor=Mahogany,linkcolor=Mahogany,urlcolor=Mahogany,filecolor=Mahogany]{hyperref}
\usepackage[capitalize]{cleveref}
\usepackage[mathscr]{euscript}
\usepackage[margin=1.25in]{geometry}
\usepackage{tikz,tikz-cd}

\newtheorem{theorem}{Theorem}[section]
\newtheorem*{theorem*}{Theorem}
\newtheorem{lemma}[theorem]{Lemma}
\newtheorem{proposition}[theorem]{Proposition}

\newtheorem*{corollary*}{Corollary}

\newtheorem{atheorem}{Theorem}

\newtheorem{innercustomgeneric}{\customgenericname}
\providecommand{\customgenericname}{}
\newcommand{\newcustomtheorem}[2]{%
  \newenvironment{#1}[1]
  {%
   \ifdefined\crefalias\crefalias{innercustomgeneric}{#2}\fi
   \renewcommand\customgenericname{#2}%
   \renewcommand\theinnercustomgeneric{##1}%
   \innercustomgeneric
  }
  {\endinnercustomgeneric}%
  \ifdefined\crefname\crefname{#2}{#2}{#2s}\fi
}
\newcustomtheorem{customthm}{Theorem}
\newtheorem{acorollary}[atheorem]{Corollary}

\newcustomtheorem{customconj}{Conjecture}
\theoremstyle{definition}
\newtheorem{definition}[theorem]{Definition}
\newtheorem*{definition*}{Definition}

\newtheorem{notation}[theorem]{Notation}
\newtheorem*{notation*}{Notation}

\theoremstyle{remark}

\newtheorem*{example*}{Example}
\newtheorem*{remark*}{Remark}

\newtheorem{remark}[theorem]{Remark}

\renewcommand{\rm}[1]{{\mathrm{#1}}}

\newcommand{\ol}[1]{{\overline{#1}}}

\newcommand{\fr}[1]{{\mathfrak{#1}}}
\newcommand{\bb}[1]{{\mathds{#1}}}

\providecommand{\GL}{\rm{GL}}
\providecommand{\Span}{\rm{span}}
\providecommand{\St}{\rm{St}}
\providecommand{\Aut}{\rm{Aut}}

\providecommand{\Out}{\rm{Out}}
\providecommand{\Z}{\mathds{Z}}
\providecommand{\Q}{\mathds{Q}}
\providecommand{\T}{\mathcal{T}}

\AtBeginDocument{%
	\def\MR#1{}
}

\title{The Morita classes are nonzero}
\author{Alexander Kupers}
\address{Department of Computer and Mathematical Sciences, University of Toronto Scarborough, 1265 Military Trail, Toronto, ON M1C 1A4, Canada}
\email{a.kupers@utoronto.ca}
\author{Jeremy Miller}
\address{Department of Mathematics, Purdue University, 150 North University, West Lafayette 47907, United States}
\email{jeremykmiller@purdue.edu}

\author{Peter Patzt}
\address{University of Oklahoma, Department of Mathematics, 601 Elm Av, Norman OK, 73019, USA}
\email{ppatzt@ou.edu}

\date{\today}

\begin{document}

\begin{abstract}We prove that all Morita classes $\mu_k \in H_{4k}(\rm{Aut}(F_{2k+2});\bb{Q})$ are nonzero using a novel description of the Hopf algebra structure on Steinberg homology.  We prove analogous results with twisted coefficients. Our results disprove a case of a finiteness conjecture of Kontsevich.\end{abstract}

\maketitle

\vspace{-.5cm}

\tableofcontents

\vspace{-.5cm}

\section{Introduction}

Borinsky--Vogtmann proved that the Euler characteristic of $\Out(F_n)$ grows super-exponentially \cite[Theorem 1.1]{BorinskyVogtman}, so $\Aut(F_n)$ and $\Out(F_n)$ have an enormous amount of rational homology. Despite the ubiquity of homology, no explicit infinite family of nonzero rational homology classes in positive degrees had been found. A similar situation for mapping class groups was called the \emph{dark matter problem} by Farb (see e.g.~Margalit \cite[Problem 9.2]{MargalitProblems}).  

In \cite[p.~391]{Morita}, Morita constructed classes in $H_{4k}(\rm{Out}(F_{2k+2});\bb{Q})$ for $k \geq 1$, which lift to classes $\mu_k \in H_{4k}(\rm{Aut}(F_{2k+2});\bb{Q})$ that we refer to as \emph{Morita classes}; see \cref{def:morita-class} for a definition suited to our arguments. It is known that the first three Morita classes $\mu_1,\mu_2,\mu_3$ are nonzero (Morita \cite[Proposition 6.13]{Morita}, Conant--Vogtmann \cite[Proposition 5.4]{ConantVogtmann}, and Gray \cite[Theorem 3.3.4]{Gray}). It is a well-known conjecture that all of the Morita classes are nonzero (see e.g. Morita \cite[Conjecture 9]{MoritaCoh} or Church--Farb--Putman \cite[Question 15]{ChurchFarbPutman}). Our main result confirms this conjecture.

\begin{atheorem}\label{athm:main} For $k \geq 2$, the images of the Morita classes $\mu_k \in H_{4k}(\rm{Aut}(F_{2k+2});\bb{Q})$ under the map induced by the abelianisation map $\rm{ab}\colon\Aut(F_{2k+2}) \to \GL_{2k+2}(\Z)$ are nonzero. In particular, for $k \geq 1$, the Morita classes and their images in $H_{4k}(\rm{Out}(F_{2k+2});\bb{Q})$ are nonzero.\end{atheorem}

\begin{remark*}This answers a question of Bridson and Vogtmann \cite[Question 32]{BridsonVogtmann} on the image of $\mu_2$ in $H_8(\GL_6(\Z);\Q)$ and the analogous problem of Morita for general $\mu_k$ \cite[Problem 37]{MoritaCoh}. We note the image of $\mu_1$ vanishes in $H_{4}(\GL_{4}(\Z);\bb{Q})$ and in fact $H_{4}(\GL_{4}(\Z);\bb{Q}) = 0$ by work of Lee and Szczarba \cite[Theorem 2]{K45}.\end{remark*}
 
To prove Theorem \ref{athm:main}, we interpret the image of the Morita classes as obtained by a push-pull construction from a span \[\rm{GL}_{2k}(\bb{Z}) \times \rm{GL}_2(\bb{Z}) \longleftarrow P_{2k,2}(\bb{Z}) \longrightarrow \rm{GL}_{2k+2}(\bb{Z})\] with $P_{2k,2}(\bb{Z})$ a certain parabolic subgroup, and pair it against a cohomology class obtained under Bieri--Eckmann duality by a product in homology with Steinberg coefficients. This reduces the result to a computation in the Hopf algebra structure of Ash and the last two authors \cite{AMP}. 

Using the same techniques we also prove a nonvanishing result with twisted coefficients. Let $\Q^{\det}$ denote the determinant representation of $\GL_n(\bb{Z})$ and $\Aut(F_n)$.

\begin{atheorem} \label{odd}
    For $k\geq 2$, there exist classes $\alpha_k \in H_{2k-1}(\Aut(F_{2k});\Q^{\det})$ that map nontrivially to $H_{2k-1}(\GL_{2k}(\Z);\Q^{\det})$.
\end{atheorem}

A consequence of Theorem \ref{athm:main} is the resolution of a case of a conjecture by Kontsevich. In \cite{Kontsevich}, he defined the notions of even and odd graph homology associated to the commutative, associative, and Lie cyclic operads. In his main theorem \cite[Theorem 1.1]{Kontsevich}, he relates the latter two even graph homology groups to the cohomology groups of the mapping class groups of surfaces and to the cohomology groups of $\Out(F_n)$, respectively. In all six cases, he made finite-dimensionality conjectures about these graph homology groups \cite[Section 7C]{Kontsevich}. 

In the even Lie case, Kontsevich's finite-dimensionality conjecture can be reformulated as the statement 
\[H_{2n-3-i}(\Out(F_n);\Q) = 0 \text{ for }n\gg i;\] see e.g.~Brun--Willwacher \cite[Theorem 5]{BrunWillwacher}.
This can be interpreted as the eventual vanishing of the codimension-$i$ homology using that $\Out(F_n)$ has virtual cohomological dimension is $2n-3$ by work of Culler and Vogtmann \cite{CV}. Our Theorem \ref{athm:main} show this conjecture is false for $i=1$.

\begin{acorollary}
For $i=1$, $H_{2n-3-i}(\Out(F_n);\Q) \neq 0$ for all even $n\ge 2$.    
\end{acorollary}

The odd Lie case of Kontsevich's conjecture is equivalent to the statement $H_{2n-3-i}(\Out(F_n);\Q^{\det} ) = 0$ for $n \gg i$. This is not contradicted by Theorem \ref{odd}, and still open to the knowledge of the authors.

    \begin{remark*}The even and odd commutative cases of Kontsevich's conjecture true;  see Campos--Idrissi--Lambrechts--Willwacher    
    \cite[Lemma A.10]{CILW} and Willwacher \cite[Theorem 1.1]{Willwacher}. The even associative case is false by work of Chan, Galatius, and Payne \cite[Theorem 1.1]{CGP}, who proved the codimension-1 homology of mapping class group does not eventually vanish.
\end{remark*}

\begin{notation*}Throughout the remainder of this paper we will work rationally, e.g.~take (co)homology with $\bb{Q}$-coefficients and rationalise the Steinberg modules, and suppress this from the notation.
\end{notation*}

\subsection*{Acknowledgments} AK acknowledges the support of the Natural Sciences and Engineering Research Council of Canada (NSERC) [funding reference number 512156 and 512250]. JM acknowledges the support of National Science Foundation (NSF) grant DMS2504473 as well as a Simons Foundation Travel Support for Mathematicians grant. PP acknowledges the support of National Science Foundation (NSF) grant DMS2202943 as well as a Simons Foundation Travel Support for Mathematicians grant. We thank Christian Kremer, Daniil Rudenko, Ismael Sierra, Dan Petersen, Karen Vogtmann, Thomas Willwacher, and Jennifer Wilson for helpful conversations. We thank Urshita Pal for helpful discussions related to \cref{s=Delta} and Sam Payne for pointing out that our results contradict Kontsevich's finite-dimensionality conjecture.  

\section{Bieri--Eckmann duality}

In this section, we give a short recollection of Bieri--Eckmann duality \cite{BieriEckmann} and provide some results that we will need in the subsequent sections.

\begin{definition}
    A (discrete) group $G$ is a \emph{rational (Bieri--Eckmann) duality group of dimension $n$} if there is a $\Q[G]$-module $D$ and an element $x\in H_n(G;D)$ such that the cap product
    \[ x \cap - \colon H^i(G;A) \longrightarrow H_{n-i}(G;A \otimes D)
    \]
    is an isomorphism for all $\Q[G]$-modules $A$ and $i\in \Z$. In this case, $D$ is called the \emph{(rational) dualising module} of $G$.
\end{definition}

It is immediate that $n$ is the rational cohomological dimension of $G$ and that $D$ is uniquely determined by the isomorphism $D \cong H^n(G;\Q[G])$. For short exact sequences of groups, Bieri--Eckmann prove the following:

\begin{theorem}[{\cite[Theorem 3.5]{BieriEckmann}}] \label{thm:be-ses}
    Consider a short exact sequence of groups
    \[ 1 \longrightarrow K \longrightarrow G \longrightarrow Q \longrightarrow 1.\]
    Assume that $K$ and $Q$ are rational duality groups of dimensions $d$ and $d'$ with dualising modules $D$ and $D'$, respectively. Then $G$ is a rational duality group of dimension $d+d'$ with dualising module $D \otimes D'$, where $G$ acts on $D \cong H^d(K;\Q[K])$ via conjugation and on $D'$ through $G \to Q$.
\end{theorem}

We may define a fibre integration map from the cohomology of $G$ to the cohomology of $Q$ by using Bieri--Eckmann duality twice:

\begin{definition}\label{def:fibreint}
    Consider a short exact sequence of groups
    \[ 1 \longrightarrow K \longrightarrow G \stackrel{q}\longrightarrow Q \longrightarrow 1.\]
    Assume that $K$ and $Q$ are rational duality groups of dimensions $d$ and $d'$ with dualising modules $D$ and $D'$, respectively. Let $A$ be a $\Q[Q]$-module. Define the \emph{fibre integration map} 
    \[{\textstyle\int_q}\colon H^i(G;A) \longrightarrow H^{i-d}(Q;D_K \otimes A)\] as the dashed left vertical map in
\[\begin{tikzcd} 
H^{i}(G;A) \dar[dashed,swap]{\int_q}   \rar{\cong}[swap]{\rm{BE}} &[10pt] H_{d+d'-i}(G;D \otimes D' \otimes A )\dar{q_*} \\[-5pt]
  H^{i-d}(Q;D_K \otimes A) \rar{\cong}[swap]{\rm{BE} } & 
 H_{d+d'-i}(Q;D_K \otimes D' \otimes A).
 \end{tikzcd}
\]
\end{definition}

Its restriction from $H^d(G)$ to $H^d(K)$ can be described via the fibre integration map as described in the following lemma.

\begin{lemma}\label{lem:be-rel-pullback} Consider a short exact sequence of groups 
    \[ 1 \longrightarrow K \stackrel{\rm{inc}}\longrightarrow G \stackrel{q}\longrightarrow Q \longrightarrow 1.\]
    Assume that $K$ and $Q$ are rational duality groups of dimensions $d$ and $d'$ with dualising modules $D$ and $D'$, respectively. Let $A$ be a $\Q[Q]$-module.  Then 
\[\begin{tikzcd} H^d(G;A) \dar[swap]{\int_q} \rar{\rm{inc}^*} & H^d(K) \otimes A   \\[-5pt]
H^{0}(Q;D_K\otimes A) \rar{\rm{BE}} & H^0(Q;H^d(K) \otimes A)\uar[hook]\end{tikzcd}\]
commutes up to a sign.
\end{lemma}

\begin{proof}
Consider the two Lyndon--Hochschild--Serre spectral sequences
\begin{align*}E^{pq}_2 = H^p(Q; H^q(K)\otimes A) &\implies H^{p+q}(G;A) \\
\ol{E}^2_{pq} = H_p(Q;H_q(K;D) \otimes D'\otimes A) &\implies H_{p+q}(G;D\otimes D' \otimes A).\end{align*}
We first observe that $\rm{inc}^*$ agrees with the composition
\[ H^d(G;A) \xrightarrow{\rm{edge}} H^0(Q;H^d(K) \otimes A) \xhookrightarrow{\phantom{edge}} H^d(K) \otimes A  \]
and $q_*$ is given by 
\[H_{d'}(G;D \otimes D' \otimes A)\xrightarrow{\ol{\rm{edge}}} H_{d'}(Q;H_0(K;D) \otimes D'\otimes A), \]
where $\rm{edge}$ and $\ol{\rm{edge}}$ are the edge maps of the above two spectral sequences, respectively. (This does not require the groups to be duality groups.) It thus suffices to show the following square commutes
\[\begin{tikzcd} H^d(G;A) \dar[swap]{\rm{BE}} \rar{\rm{edge}} & H^0(Q;H^d(K) \otimes A) \dar{\rm{BE}} \\[-5pt]
H_{d'}(G;D\otimes D' \otimes A) \rar{\ol{\rm{edge}}} & H_{d'}(Q;H_0(K;D)\otimes D' \otimes A).\end{tikzcd}\]
As the Bieri--Eckmann isomorphisms are defined by the cap products, \cite[Lemma 1.3.2]{Bieri} (applied with $N=K$, $G/N = Q$, $A = A$, $C = D \otimes D'$, restricting the top-left corner to $\bb{Q}\{x\} \otimes E^{p,q}_2$ with $x$ the generator of $\smash{\ol{E}{}^2_{d',d}}$)  states that
\[E^{pq}_2 = H^p(Q;H^q(K) \otimes A) \overset{\rm{BE}}\longrightarrow H_{d'-p}(Q; H_{d-q}(K;D) \otimes D' \otimes A)  = \ol{E}^2_{d'-p,d-q}  \]
gives rise to an isomorphism of spectral sequences when multiplied with the correct sign. Moreover, by \cite[Lemma 1.3.1]{Bieri} (applied to same data as above), this isomorphism on the spectral sequences converges to the Bieri--Eckmann isomorphism on the abutment. Thus the Bieri--Eckmann isomorphisms commute with the edge morphisms.
\end{proof}

\section{A geometric description of the coproduct in Steinberg homology} 
In this section, we give a geometric model of a coproduct introduced in \cite[Definition 2.7]{AMP}. Brown, Galatius, Chan, and Payne \cite{BCGP} independently constructed a coproduct using different methods.

\subsection{The coproduct in Steinberg homology} We start with some recollections. For $V$ an $n$-dimensional $\Q$-vector space, recall that the Tits building $\T(V)$ is the order complex of the poset of proper nonzero subspaces of $V$ ordered by inclusion. The Solomon--Tits Theorem says that $\T(V)$ is $(n-2)$-spherical and we define the \emph{Steinberg module} of $\GL(V)$ as $\St(V) \coloneq \widetilde H_{n-2}(\T(V))$ for $n>0$ and define it to be $\Q$ for $n=0$. The Steinberg module is generated by apartment classes: let $v_1 , \ldots , v_n$ be a basis of $V$, then there is an associated \emph{apartment class} in $\St(V)$ defined by the formula (viewed as a top-degree reduced cellular chain $\widetilde C_{n-2}(\T(V))$)
\[
  [\![v_1,\ldots,v_n]\!] \coloneq \sum_{\sigma \in \fr{S_n}} (-1)^{\sigma} \Span(v_{\sigma (1)}) < \Span(v_{\sigma (1)},v_{\sigma (2)})  < \ldots < \Span(v_{\sigma (1)},\ldots, v_{\sigma (n-1)}).  
\]

\begin{definition}
Let $V$ be an $n$-dimensional $\Q$-vector space and $\chi\colon \Z^\times \to \Q^\times$ a group homomorphism, i.e.~either the inclusion or the trivial map. We denote by $\Q^\chi$ the trivial $\GL(V)$-module if $\chi$ is the trivial map and $\Lambda^n V$ if $\chi$ is the inclusion. Let $\St^{\chi}(V)$ be $\St(V)\otimes \Q^\chi$. If $\chi$ is the inclusion, we will also write $\St^{\det}(V)$ for $\St^{\chi}(V)$ and $\Q^{\det}$ for $\Q^\chi$, and if $\chi$ is the trivial map, we simply write $\St(V)$ for $\St^{\chi}(V)$ and $\Q$ for $\Q^\chi$.\end{definition}

\begin{notation}By abuse of notation, we also use the notation $[\![v_1,\ldots,v_n]\!]$ for the element of $\St^{\det}(V)$ given by tensoring the same-named apartment class with $v_1 \wedge \cdots \wedge v_n$.\end{notation}

We now describe the coproduct map. Let $\rm{Sh}_{a,b} \subseteq \fr{S}_n$ denote the subset of the symmetric group given by $a,b$-shuffles. For $a+b=n$, we define the map (well-defined by \cite[Lemma 2.10]{AMP})
\begin{align*}\Delta_{a,b} \colon \St^{\chi}(\Q^n)  &\longrightarrow  \bigoplus_{V \subseteq \Q^n , \,\dim(V)=a} \St^{\chi}(V)   \otimes \St^{\chi}(\Q^n/V) \\
[\![v_1,\ldots,v_n]\!] &\longmapsto \sum_{\sigma \in \rm{Sh}_{a,b}} (-1)^\sigma [\![v_{\sigma(1)},\ldots,v_{\sigma(a)}]\!] \otimes [\![v_{\sigma(a+1)},\ldots,v_{\sigma(n)}]\!].\end{align*}
We next introduce the following parabolic subgroup of $\rm{GL}_n(\bb{Z})$ 
\[P_{a,b}(\bb{Z}) \coloneq \left\{\begin{bmatrix} A & B \\ 0 & C \end{bmatrix} \,\,\middle \vert \,\,\text{with $A \in \rm{GL}_{a}(\bb{Z})$, $B \in \rm{Mat}_{a,b}(\bb{Z})$, $C \in \rm{GL}_{b}(\bb{Z})$}\right\}.\] 
There is an action of $P_{a,b}(\Z)$ on
$\St^{\chi}(\Q^a) \otimes \St^{\chi}(\Q^b)$ through the projection map \[\rm{pr} \colon P_{a,b}(\Z) \longrightarrow \GL_a(\Z) \times \GL_b(\Z),\] which appears when applying Shapiro's lemma to obtain the isomorphism
\[H_\ast \Big(\GL_n(\Z) ; \bigoplus_{V \subseteq \Q^n, \, \dim(V)=a} \St^{\chi}(V)   \otimes \St^{\chi}(\Q^n/V) \Big) \cong H_*\Big(P_{a,b}(\Z) ; \St^{\chi}(\Q^a) \otimes \St^{\chi}(\Q^b)\Big).\]
Since the action factors over the projection, we also have an induced map 
\[H_*(P_{a,b}(\Z) ; \St^{\chi}(\Q^a) \otimes \St^{\chi}(\Q^b)) \longrightarrow H_*(\GL_a(\Z) \times \GL_b(\Z) ; \St^{\chi}(\Q^a) \otimes \St^{\chi}(\Q^b)).\]
The composition of $\Delta_{a,b}$ with these maps and the K\"unneth map endows the bigraded $\bb{Q}$-vector space 
\[\bigoplus_{n,i\ge 0} H_i(\GL_n(\Z) ; \St^{\chi}(\Q^n))\] 
with a coproduct, as is verified in \cite[Lemma 2.10]{AMP}. In that paper, an explicit resolution of $\St(\Q^n)$ due to Lee and Szczarba \cite{LS} is used  to compute group homology but the maps on homology do not depend on the choice of resolution. Note that $H_*(\GL_n(\Z) ; \St^{\det}(\Q^n))$ is trivial for $n$ odd by the ``centre kills'' trick.

\subsection{A Bieri--Eckmann duality description of the coproduct}

We now give an alternative description of this coproduct through Bieri--Eckmann duality \cite{BieriEckmann}, see \cref{thm:geometric-coproduct}. Recall that the coproduct involves two steps: 
\begin{enumerate}[(1)]
    \item the coproduct map on dualising modules followed by Shapiro's lemma, 
    \item\label{item:pr} the map induced by the projection $\rm{pr}$.
\end{enumerate}

We start with a Bieri--Eckmann dual description of (1). Let $V$ be an $n$-dimension $\Q$-vector space. Let $\St^{\rm{or}}(V)$ denote $\St(V)$ if $n$ is odd and $\St^{\det}(V)$ if $n$ is even.  Work of Borel--Serre \cite{BoSe}  implies that $\GL_n(\Z)$ and $P_{a,b}(\Z)$ are rational duality groups of dimension $\smash{\nu_n \coloneq {n \choose 2}}$. The rational dualising module of $\GL_n(\Z)$ is $\St^{\rm{or}}(\Q^n)$; see Putman--Studenmund \cite[Theorem C]{PutmanStudenmund}. The rational dualising module of $P_{a,b}(\Z)$ is $\St^{\det}(\Q^a) \otimes \St^{\det}(\Q^b)$ if $a$ and $b$ have the same parity, which may be seen applying \cref{thm:be-ses} to the short exact sequence 
\[1 \longrightarrow \bb{Z}^{ab} \longrightarrow P_{a,b}(\bb{Z}) \overset{\rm{pr}}\longrightarrow \rm{GL}_{a}(\bb{Z}) \times \rm{GL}_{b}(\bb{Z}) \longrightarrow 1,\]
in which the left term $\bb{Z}^{ab}$ is a Poincar\'e duality group and hence rational duality group with rational dualising module $\bb{Q}$ and $\rm{GL}_{a}(\bb{Z}) \times \rm{GL}_{b}(\bb{Z})$ is also a rational duality group with dualising module $\St^{\rm{or}}(\Q^a) \otimes \St^{\rm{or}}(\Q^b)$. As $(A,B) \in \GL_a(\Z) \times \GL_b(\Z)$ acts on the dualising module of $\Z^{ab}$ by $\det(A)^b\cdot \det(B)^{-a}$, the middle term has dualising module
\[ \St^{\rm{or}}(\Q^a) \otimes \St^{\rm{or}}(\Q^b) = \St^{\det}(\Q^a) \otimes \St^{\det}(\Q^b) \]
if $a$ and $b$ are both even, and
\[(\Q^{\det} \otimes\St^{\rm{or}}(\Q^a)) \otimes (\Q^{\det} \otimes \St^{\rm{or}}(\Q^b)) = \St^{\det}(\Q^a) \otimes \St^{\det}(\Q^b) \]
if $a$ and $b$ are both odd.

\smallskip

Take $n=a+b$, let $\widetilde C_i$ denote reduced cellular chains, and let $*$ denote the simplicial join. In \cite{Reeder}, Reeder defined the following map:
\begin{align*} s^{a,b} \colon \widetilde C_p(\T(\Q^n)) &\longrightarrow \widetilde C_{p-1}(\T(\Q^a) * \T(\Q^b) )\\
(V_0< \ldots <V_p) &\longmapsto \begin{cases}
(V_0< \ldots <V_{j-1}) * ((V_{j+1} \cap \Q^b)< \ldots <  (V_{p} \cap \Q^b)), & \text{if } V_j = \Q^a, \\
0, & \text{if }V_j \neq \Q^a \text{ for all }j.
\end{cases}\end{align*}
Here, we view $\Q^n$ as $\Q^a \oplus \Q^b$, where $\Q^a\subseteq \Q^n$ is the subspace of vectors whose last $b$ coordinates are zero and $\Q^b\subseteq \Q^n$ is the subspace of vectors whose first $a$ coordinates are zero. Reeder proved that $s^{a,b}$ induces a map on the Steinberg modules $\St^{\det}(\Q^n) \to \St^{\det}(\Q^a) \otimes \St^{\det}(\Q^b)$ if $n$ is even, and that this map is dual to the usual restriction map on cohomology \cite[Proposition 4.4]{Reeder}.

\begin{theorem}[Reeder] \label{lem:reeder} Let $n$ be even and $a+b =n$. Let $\rm{inc} \colon P_{a,b}(\Z) \to \GL_n(\Z)$ be the inclusion. There is a commutative square
\[\begin{tikzcd}H^i(\rm{GL}_n(\bb{Z}); \Q^\chi) \rar{\rm{inc}^*} \dar{\rm{BE}}[swap]{\cong} &[10pt] H^i(P_{a,b}(\Z);\Q^\chi) \dar{\rm{BE}}[swap]{\cong} \\[-5pt]
H_{\nu_n-i}(\rm{GL}_{n}(\bb{Z}); \St^{\det}(\Q^n)\otimes \Q^\chi) \rar{s^{a,b}_*} & H_{\nu_n-i}(P_{a,b}(\Z); \St^{\det}(\Q^a) \otimes \St^{\det}(\Q^b) \otimes \Q^\chi).\end{tikzcd}\]
\end{theorem}

We now compare Reeder's map $s^{a,b}$ and the map $\Delta_{a,b}$.

\begin{lemma} \label{s=Delta} After identifying $\Q^b$ with $\Q^n/ \Q^a$, the projection of 
\[\Delta_{a,b} \colon \St^{\det}(\Q^n)  \longrightarrow  \bigoplus_{V \subseteq \Q^n , \,\dim(V)=a} \St^{\det}(V)   \otimes \St^{\det}(\Q^n/V)\] 
to the summand associated to $V=\Q^a$ agrees with $s^{a,b} \colon\St^{\det}(\Q^n)\to \St^{\det}(\Q^a)  \otimes \St^{\det}(\Q^b) $.
\end{lemma}

\begin{proof}
    It suffices to check this claim on apartments, viewing $\Q^n$ as $\Q^a \oplus \Q^b$.    
    Observe that $s^{a,b}([\![v_1,\ldots,v_n]\!])=0$ if $v_1,\ldots,v_n$ does not contain a basis of $\Q^a$. Thus we assume $v_1,\ldots,v_n$ \emph{does} contain a basis of $\Q^a$, and then there is a unique $(a,b)$-shuffle $\sigma$ with $v_{\sigma(1)},\ldots,v_{\sigma(a)}$ a basis of $\Q^a$. If $\tau \in \fr{S}_n$ does not equal $\sigma$ modulo $\fr{S}_a \times \fr{S}_b$, then the flag \[ \Span(v_{\tau (1)}) < \Span(v_{\tau (1)},v_{\tau (2)})  < \ldots < \Span(v_{\tau (1)},\ldots, v_{\tau (n-1)}) \] is mapped to $0$ by $s^{a,b}$.
    If $\tau \in \fr{S}_n$ is equal to $\sigma$ modulo $\fr{S}_a \times \fr{S}_b$, then the flag \[ \Span(v_{\tau (1)}) < \Span(v_{\tau (1)},v_{\tau (2)})  < \ldots < \Span(v_{\tau (1)},\ldots, v_{\tau (n-1)}) \] is mapped by $s^{a,b}$ to \[\Big( \Span(v_{\tau (1)}) < \ldots < \Span(v_{\tau (1)},\ldots, v_{\tau (a-1)})    \Big)* \Big( \Span(v_{\tau (a+1)}) < \ldots < \Span(v_{\tau (a+1)},\ldots, v_{\tau (n-1)})    \Big). \]
It follows from this and the definition of apartment classes that  we have an equality
     \[s^{a,b}([\![v_1,\ldots,v_n]\!])=(-1)^\sigma [\![v_{\sigma(1)},\ldots,v_{\sigma(a)} ]\!] \otimes [\![v_{\sigma(a+1)},\ldots,v_{\sigma(n)} ]\!] \] in this case. This is exactly the projection of $\Delta_{a,b}([\![v_1,\ldots,v_n]\!])$.
\end{proof}

We next describe the fibre integration map of the projection map. 
Applying \cref{def:fibreint} to the present situation, we get the commutative square
\[\begin{tikzcd} 
H^{\nu_n-*}(P_{a,b}(\bb{Z})) \dar[swap]{\int_\rm{pr}}   \rar{\cong}[swap]{\rm{BE}} &[10pt] H_*(P_{a,b}(\bb{Z});\rm{St}^{\det}(\Q^a) \otimes \rm{St}^{\det}(\Q^b) )\dar{\rm{pr}_*} \\[-5pt]
  H^{\nu_a-*}(\rm{GL}_{a}(\bb{Z}))  \otimes H^{\nu_b-*}(\rm{GL}_{b}(\bb{Z})) \rar{\cong}[swap]{\rm{BE} \otimes \rm{BE}} & 
 H_*(\rm{GL}_{a}(\bb{Z});\rm{St}^{\det}(\Q^a)) \otimes H_*(\rm{GL}_{b}(\bb{Z});\rm{St}^{\det}(\Q^b)).
 \end{tikzcd}
\]
if $a$ and $b$ are both even, and the commutative square
\[\begin{tikzcd} 
H^{\nu_n-*}(P_{a,b}(\bb{Z});\Q^{\det}) \dar[swap]{\int_\rm{pr}}   \rar{\cong}[swap]{\rm{BE}} &[10pt] H_*(P_{a,b}(\bb{Z});\rm{St}(\Q^a) \otimes \rm{St}(\Q^b) )\dar{\rm{pr}_*} \\[-5pt]
  H^{\nu_a-*}(\rm{GL}_{a}(\bb{Z}))  \otimes H^{\nu_b-*}(\rm{GL}_{b}(\bb{Z})) \rar{\cong}[swap]{\rm{BE} \otimes \rm{BE}} & 
 H_*(\rm{GL}_{a}(\bb{Z});\rm{St}(\Q^a)) \otimes H_*(\rm{GL}_{b}(\bb{Z});\rm{St}(\Q^b)).
 \end{tikzcd}
\]
if $a$ and $b$ are both odd. Thus, by construction, $\int_\rm{pr}$ is Bieri--Eckmann dual to the projection map of \eqref{item:pr}. In conclusion, we get the following theorem.

\begin{theorem}\label{thm:geometric-coproduct} Let $n$ be even and $a+b=n$. 
\begin{enumerate}[\noindent (i)]
    \item If $a$ and $b$ are both even, there is a commutative square
\[\begin{tikzcd}[column sep=4em]
H^{\nu_n-*}(\rm{GL}_n(\bb{Z})) \rar{\int_\rm{pr} \circ \ \rm{inc}^*} \dar{\rm{BE}}[swap]{\cong} & H^{\nu_a-*}(\rm{GL}_{a}(\bb{Z})) \otimes H^{\nu_b-*}(\rm{GL}_{b}(\bb{Z})) \dar{\rm{BE} \otimes \rm{BE}}[swap]{\cong} \\[-5pt]
H_*(\rm{GL}_{a+b}(\bb{Z});\rm{St}^{\det}(\Q^n)) \rar{\Delta_{a,b}}  &[10pt] H_*(\rm{GL}_{a}(\bb{Z});\rm{St}^{\det}(\Q^a)) \otimes H_*(\rm{GL}_{b}(\bb{Z});\rm{St}^{\det}(\Q^b)) 
.\end{tikzcd}\]
\item If $a$ and $b$ are both odd, there is a commutative square
\[\begin{tikzcd}[column sep=4em]
H^{\nu_n-*}(\rm{GL}_n(\bb{Z});\Q^{\det}) \rar{\int_\rm{pr} \circ \ \rm{inc}^*} \dar{\rm{BE}}[swap]{\cong} & H^{\nu_a-*}(\rm{GL}_{a}(\bb{Z})) \otimes H^{\nu_b-*}(\rm{GL}_{b}(\bb{Z})) \dar{\rm{BE} \otimes \rm{BE}}[swap]{\cong} \\[-5pt]
H_*(\rm{GL}_{a+b}(\bb{Z});\rm{St}(\Q^n)) \rar{\Delta_{a,b}}  &[10pt] H_*(\rm{GL}_{a}(\bb{Z});\rm{St}(\Q^a)) \otimes H_*(\rm{GL}_{b}(\bb{Z});\rm{St}(\Q^b)) 
.\end{tikzcd}\]
\end{enumerate}
\end{theorem}

\section{Unipotent abelian cycles in $\Aut(F_n)$}

In this section, we describe a procedure for constructing nonzero homology classes for $\Aut(F_n)$. 

\begin{definition}\label{def:unipabcyc}
    Let $n$ be even and $a+b = n$. We call a map $\alpha\colon \Z^{ab} \to \Aut(F_n)$ a \emph{unipotent abelian cycle} if the composition 
    \[ \Z^{ab} \stackrel{\alpha}\longrightarrow \Aut(F_n) \stackrel{\rm{ab}}\longrightarrow \GL_n(\Z) \]
    is equal to the composition $\Z^{ab} \to P \to \GL_n(\Z)$, where $P$ is a subgroup of $\GL_n(\Z)$ that stabilizes a rank-$a$ summand $M$ of $\Z^n$ and $\Z^{ab} \to P$ is an injection whose image is the subgroup of invertible matrices that fix pointwise the elements of $M$ and $\Z^n/M$.
\end{definition}

A unipotent abelian cycle often gives a nonzero homology class of $\Aut(F_n)$:

\begin{theorem}\label{thm:unipabcyc}
    Let $n$ be even and $a+b = n$, and assume that $a \equiv 0,3 \pmod 4$ if $a=b$. Let $\alpha\colon \Z^{ab} \to \Aut(F_n)$ be a unipotent abelian cycle. 
    \begin{enumerate}[\noindent (i)]
        \item 
    If $a$ and $b$ are both even, then the following composition of induced maps is injective:
    \[ H_{ab}(\Z^{ab}) \stackrel{\alpha_*}\longrightarrow H_{ab}(\Aut(F_n)) \stackrel{\rm{ab}_*}\longrightarrow H_{ab}(\GL_n(\Z)).\]
\item 
    If $a$ and $b$ are both odd, then the following composition of induced maps is injective:
    \[ H_{ab}(\Z^{ab}) \stackrel{\alpha_*}\longrightarrow H_{ab}(\Aut(F_n);\Q^{\det}) \stackrel{\rm{ab}_*}\longrightarrow H_{ab}(\GL_n(\Z);\Q^{\det}).\]
    \end{enumerate}
\end{theorem}

Using a change of basis, we may assume that $P=P_{a,b}(\bb{Z})$. By definition of unipotent abelian cycles, \cref{thm:unipabcyc} follows immediately from the following statements:

\begin{proposition}\label{prop:uniabcyclesnonzero}
Let $n$ be even and $a+b = n$, and assume that $a \equiv 0,3 \pmod 4$ if $a=b$:
\begin{enumerate}[\noindent (i)]
    \item If $a$ and $b$ are both even, then the following composition of induced maps is injective:
    \[H_{ab}(\Z^{ab}) \longrightarrow H_{ab}(P_{a,b}(\Z)) \longrightarrow H_{ab}(\GL_n(\Z)).\]
    \item 
    If $a$ and $b$ are both odd, then the following composition of induced maps is injective:
    \[H_{ab}(\Z^{ab}) \longrightarrow H_{ab}(P_{a,b}(\Z);\Q^{\det}) \longrightarrow H_{ab}(\GL_n(\Z);\Q^{\det}).\]
    \end{enumerate}
\end{proposition}

Before proving this proposition, let us prove a useful lemma.

\begin{lemma}\label{lem:fundclass} 
Let $n$ be even and $a+b = n$.
\begin{enumerate}[\noindent (i)]
    \item \label{enum:fundclass-i} If $a$ and $b$ are both even, the image of the fundamental class through  \[H_{ab}(\Z^{ab}) \longrightarrow H_{ab}(P_{a,b}(\Z)) \]
    is given by applying $({\textstyle \int}_\rm{pr})^\vee$ to the generator $1 \otimes 1 \in H_0(\rm{GL}_{a}(\bb{Z})) \otimes H_0(\rm{GL}_b(\bb{Z}))$ up to a sign.
    \item \label{enum:fundclass-ii} If $a$ and $b$ are both odd, the image of the fundamental class through  \[H_{ab}(\Z^{ab}) \longrightarrow H_{ab}(P_{a,b}(\Z);\Q^{\det})  \]
    is given by applying $ ({\textstyle \int}_\rm{pr})^\vee$ to the generator $1 \otimes 1 \in H_0(\rm{GL}_{a}(\bb{Z})) \otimes H_0(\rm{GL}_b(\bb{Z}))$ up to a sign.
\end{enumerate}
\end{lemma}

\begin{proof}
We will apply \cref{lem:be-rel-pullback} to the short exact sequence
\[ 1 \longrightarrow \Z^{ab}\xrightarrow{\rm{inc}} P_{a,b}(\Z) \xrightarrow{\rm{pr}} \GL_a(\Z) \times \GL_b(\Z) \longrightarrow 1.\]
Note that $H^{ab}(\Z^{ab})$ is one-dimensional, and $\GL_a(\Z) \times \GL_b(\Z)$ acts on it trivially in case (1) and via the determinant in case (2).

Thus, in case \eqref{enum:fundclass-i}, we see that
\[H^{ab}(P_{a,b}(\Z)) \xrightarrow{\rm{inc}^*} H^{ab}(\Z^{ab}) \cong \Q \quad \text{and} \quad 
H^{ab}(P_{a,b}(\Z)) \xrightarrow{\int_{\rm{pr}}} H^{0}(\GL_a(\Z) \times \GL_b(\Z)) \cong \Q\]
agree up to a sign. In case \eqref{enum:fundclass-ii}, we similarly see that
\begin{align*}&H^{ab}(P_{a,b}(\Z);\Q^{\det}) \xrightarrow{\rm{inc}^*} H^{ab}(\Z^{ab}) \otimes \Q^{\det} \cong \Q \quad \text{and} \\
&H^{ab}(P_{a,b}(\Z);\Q^{\det}) \xrightarrow{\int_{\rm{pr}}} H^{0}(\GL_a(\Z) \times \GL_b(\Z); H^{ab}(\Z^{ab}) \otimes \Q^{\det}) \cong \Q\end{align*}
agree up to a sign. Dualising both observations yields the statement.
\end{proof}

\begin{proof}[Proof of \cref{prop:uniabcyclesnonzero}] By \cref{lem:fundclass}, the image of the fundamental class---in $H_{ab}(P_{a,b}(\Z))$ if $a$ and $b$ are even and in $H_{ab}(P_{a,b}(\Z);\Q^{\det})$ if $a$ and $b$ are odd---can (up to sign) be described as $ ({\textstyle \int}_\rm{pr})^\vee(1 \otimes 1)$, where $1 \otimes 1 \in H_0(\rm{GL}_{a}(\bb{Z})) \otimes H_0(\rm{GL}_b(\bb{Z}))$ is the canonical generator. Thus by \cref{thm:geometric-coproduct}, \[\langle x,\rm{inc}_* ({\textstyle \int}_\rm{pr})^\vee(1 \otimes 1) \rangle = \langle {\textstyle \int}_\rm{pr} \rm{inc}^*(x),1 \otimes 1 \rangle = \langle (\rm{BE} \otimes \rm{BE})^{-1} \Delta_{a,b}(\rm{BE}(x)),1 \otimes 1 \rangle\]
    for any $x \in H^{ab}(\rm{GL}_{n}(\Z))$ if $a$ and $b$ are even, and any $x \in H^{ab}(\rm{GL}_{n}(\Z);\Q^{\det})$ if $a$ and $b$ are odd.
    
Writing $1^\vee \in H^0(\GL_k(\Z))$ for the generator, we denote $\rm{BE}(1^\vee) \in H_{\nu_k}(\GL_k(\Z);\St^{\rm{or}}(\Q^k))$ as $t_k$. For degree reasons, the classes $t_k$ are primitive \cite[Proposition 4.2]{AMP}. We now take $x = \rm{BE}^{-1}(t_{a}t_{b})$ where $t_{a}t_{b}$ is formed with respect to the product 
\[H_{\nu_a}(\GL_a(\Z);\St^{\rm{or}}(\Q^a)) \otimes H_{\nu_b}(\GL_b(\Z);\St^{\rm{or}}(\Q^b)) \longrightarrow H_{\nu_a + \nu_b}(\GL_n(\Z);\St^{\rm{or}}(\Q^n))\]
of \cite{AMP}. Note that here we use that $a$ and $b$ have the same parity.

Then, the above evaluation is nonzero if and only if $\Delta_{a,b}(t_{a}t_{b})$ has a nonzero $t_{a}\otimes t_{b}$-component, because this is the only element taken by $(\rm{BE} \otimes \rm{BE})^{-1}$ to $1^\vee \otimes 1^\vee\in H^0(\GL_a(\Z)) \otimes H^0(\GL_b(\Z))$. As the product and coproduct assemble to a Hopf algebra structure \cite[Theorem A]{AMP} and the classes $t_{a}$ and $t_{b}$ are primitive, we compute
\[\Delta(t_{a}t_{b}) = (1 \otimes t_{a}+t_{a} \otimes 1)(1 \otimes t_{b}+t_{b} \otimes 1) = 1 \otimes t_{a}t_{b}+t_{a} \otimes t_{b}+(-1)^{(\nu_a+a)(\nu_b+b)} t_{b} \otimes t_{a} + t_{a}t_{b} \otimes 1.\]
This has a nonzero $t_{a}\otimes t_{b}$-component if $a \neq b$ (e.g.~because for $a \neq b$, $t_{a}\otimes t_{b}$ lies in a different summand of the tensor product than $t_{b}\otimes t_{a}$). Moreover, if $a=b$, then $(\nu_a +a) = (\nu_b+b)$ is even if and only if $a \equiv 0,3 \pmod 4$. This completes the proof of \cref{prop:uniabcyclesnonzero}.
%\ref{thm:nonzeroparabolic}.
\end{proof}

\section{The Morita classes} 

In this section, we apply \cref{thm:unipabcyc} to show that Morita classes are nonzero by verifying that they are represented by unipotent abelian cycles (\cref{def:unipabcyc}).

\subsection{Morita classes and their images} We start by recalling the description of the Morita classes from Conant--Hatcher--Kassabov--Vogtmann \cite[Proposition 4.4]{CHKV}. Denote the generators of the free group $F_{2k+2}$ by $x_1,\ldots,x_{2k+2}$ and consider the following $4k$ automorphisms: each of them is the identity on all generators except for
\begin{align*}\lambda_i \colon x_i &\longmapsto x_{2k+1}x_i \qquad \text{for $1 \leq i \leq 2k$}, \\
\rho_i \colon x_i &\longmapsto x_ix_{2k+2} \qquad \text{for $1 \leq i \leq 2k$}.\end{align*}
These visibly commute so they induce a homomorphism $\phi \colon \bb{Z}^{4k} \to \rm{Aut}(F_{2k+2})$.

\begin{definition}\label{def:morita-class} The \emph{Morita class} $\mu_k \in H_{4k}(\rm{Aut}(F_{2k+2}))$ is the image of the fundamental class $[\bb{Z}^{4k}] \in H_{4k}(\bb{Z}^{4k})$ under $\phi_*$.
\end{definition}

We will prove that this class is nontrivial for $k \geq 2$ (Theorem \ref{athm:main}) by observing that $\phi$ is a unipotent abelian cycle and applying \cref{thm:unipabcyc}.  Note that indeed, $\rm{ab} \circ \phi$ factors through 
\[ \Z^{4k} \longrightarrow P_{2k,2}(\bb{Z}) \longrightarrow \GL_{2k+2}(\Z)\]
and thus is a unipotent abelian cycle. Because $2k \neq 2$ for $k \ge 2$, \cref{thm:unipabcyc} implies that 
\[ H_{4k}(\Z^{4k}) \xrightarrow{\phi_*} H_{4k}(\Aut(F_{2k+2}) ) \xrightarrow{\rm{ab}} H_{4k}(\GL_{2k+2}(\Z))\]
is injective. This implies Theorem \ref{athm:main}.

\subsection{Twisted Morita classes}
A second family of easy-to-define unipotent abelian cycles
\[ \psi \colon \Z^{2k-1} \longrightarrow \Aut(F_{2k})\]
is given by the commuting automorphisms that are the identity on all generators except for
\[\lambda_i \colon x_i \longmapsto x_{2k} x_i \qquad \text{for $1 \leq i \leq 2k-1$.} \]
As $2k-1\neq 1$ for $k\ge2$ and both are odd, \cref{thm:unipabcyc} implies that
\[ H_{2k-1}(\Z^{2k-1}) \xrightarrow{\psi_*} H_{2k-1}(\Aut(F_{2k});\Q^{\det}) \xrightarrow{\rm{ab}} H_{2k-1}(\GL_{2k}(\Z);\Q^{\det})\]
is injective. Let $\alpha_k$ be the image of the fundamental class $\psi_*([\Z^{2k-1}])$. This establishes Theorem \ref{odd}.

\begin{remark}
    In \cite[Conjecture 14]{BrunWillwacher}, Brun--Willwacher ask whether different cycles that they construct in $H_{2k-1}(\Out(F_{2k});\Q^{\det})$ are nonzero. We believe these are also unipotent abelian cycles and thus nonzero, but will not prove this here. In any case, our methods do not allow us to distinguish Brun--Willwacher's classes from the classes $\alpha_k$.
\end{remark}

\bibliographystyle{amsalpha}
\bibliography{./refs}

\bigskip

\end{document}